\RequirePackage{fix-cm}
\documentclass[smallextended]{svjour3}       
\smartqed  
\usepackage{graphicx}
\usepackage{amssymb,amsmath,amsbsy}
\DeclareMathAlphabet{\mathpzc}{OT1}{pzc}{m}{it}
\newcommand{\bff}{\mathpzc{f}}
\newcommand{\mcA}{\mathcal{A}}
\newcommand{\mcB}{\mathcal{B}}
\newcommand{\mcF}{\mathcal{F}}
\newcommand{\mcL}{\mathcal{L}}
\newcommand{\ul}{\underline}

\newcommand{\citecond}[1]{\textnormal{\textbf{#1}}}
\newcommand{\hnorm}{|\hspace{-.1em}|\hspace{-.1em}|}
\newcommand{\nnsup}[2]{\sup_{0\not=#1\in#2}}
\newcommand{\les}{\lesssim}
\newcommand{\ges}{\gtrsim}
\newtheorem{notation}{Notation}
\usepackage{url}

\begin{document}

\title{A robust all-at-once multigrid method for the Stokes control problem\thanks{The research was funded by the Austrian Science Fund (FWF):
    J3362-N25.}
}
\author{Stefan Takacs}
\institute{Stefan Takacs \at
              Visiting Postdoc, Mathematical Institute, University of Oxford, United Kingdom \\
              Tel.: +44-1865-(6)15312 \\
              Fax: +44-1865-273583 \\
              \email{stefan.takacs@maths.ox.ac.uk} \\
              \email{stefan.takacs@numa.uni-linz.ac.at}           
}

\date{}

\maketitle

\begin{abstract}
In this paper we present an all-at-once multigrid method
for a distributed Stokes control problem (velocity tracking problem). 
For solving such a problem, we use the fact that the solution is
characterized by the optimality system (Karush-Kuhn-Tucker-system). The discretized optimality system
is a large-scale linear system whose condition number depends on the grid size and on
the choice of the regularization parameter forming a part of the problem. Recently,
block-diagonal preconditioners have been proposed, which allow
to solve the problem using a Krylov space method with convergence rates that are
robust in both, the grid size and the regularization parameter or cost parameter.
In the present paper, we develop an all-at-once multigrid method for a
Stokes control problem and show robust convergence, more precisely, we show that the method
converges with rates which are bounded away from one by a constant which is
independent of the grid size and the choice of the regularization or cost parameter.
\keywords{PDE-constrained optimization \and all-at-once multigrid methods \and Stokes control}
\end{abstract}

\section{Introduction}\label{sec:1}

In the present paper, we consider the following model problem
\emph{(distributed velocity tracking problem)}.
Let $\Omega\subseteq\mathbb{R}^d$ be a bounded domain with $d\in\{2,3\}$.
Find a velocity field $u\in [H^1(\Omega)]^d$, a
pressure distribution
$p\in L^2(\Omega)$ and a control (force field) $f\in [L^2(\Omega)]^d$ such that
the tracking functional
\begin{equation}\nonumber
        J(u,f) = \frac12 \|u-u_D\|_{L^2(\Omega)}^2 +
\frac{\alpha}{2} \|f\|_{L^2(\Omega)}^2
\end{equation}
is minimized subject to the Stokes equations
\begin{equation}\nonumber
\begin{aligned}
        -\Delta u + \nabla p &= f \mbox{ in } \Omega,  \\
        \nabla \cdot u &= 0 \mbox{ in } \Omega, \\
        u &= 0
        \mbox{ on } \partial \Omega. \\
\end{aligned}
\end{equation}
The cost parameter or regularization
parameter~$\alpha>0$ and the desired state (desired velocity field) $u_D\in [L^2(\Omega)]^d$ are assumed
to be given. To enforce uniqueness of the solution, we additionally require
$\int_{\Omega} p \mbox{ d}x=0$.

Here and in what follows, $L^2(\Omega)$ and~$H^1(\Omega)$ denote the 
standard Lebesgue and Sobolev spaces with associated standard
norms~$\|\cdot\|_{L^2(\Omega)}$ and~$\|\cdot\|_{H^1(\Omega)}$,
respectively.

The main goal of this work is to construct and to analyze numerical
methods that produce an approximate solution
to the optimization problem, where the computational complexity can be
bounded by the number of unknowns times a
constant which is independent of the grid level
and the choice of the parameter~$\alpha$, in particular for 
small values of~$\alpha$.

The solution of the optimization problem is characterized by the
Karush-Kuhn-Tucker-system (KKT-system).
As we are interested in good approximations of the solution, the
discretization of the KKT-system leads to a large-scale linear system. This linear system will
be solved with multigrid methods
because they are one of the fastest known methods for such problems.
Originally, multigrid methods have been designed and analyzed for
elliptic problems.
They also work well for saddle point problems (like KKT-systems)
and have gained growing interest in this area, see,
e.g.,~\cite{Borzi:Schulz:2009} and the references cited there.

The unknowns of the discretized KKT-system for a PDE-constrained
optimization problem can be partitioned into primal variables and 
dual variables. In our case, the primal variables 
consist of state variables (the velocity field~$u$ and the pressure distribution~$p$) and control 
variables (the force field~$f$). The dual variables are the Lagrange multipliers which
are introduced to incorporate the constraints.
One approach to solve such problems is to apply multigrid methods in
every step of an overall block-structured
iterative method to equations in just one of these blocks of variables.
Such methods have been proposed, e.g., in~\cite{Zulehner:2010,Kollmann:Zulehner:2012,Pearson:2012}.

Another approach, which we will follow here, is to apply the multigrid
idea directly to the (reduced or not reduced) KKT-system,
which is called an all-at-once approach. Such methods have been
proposed and discussed for the elliptic
optimal control problem, e.g., in
\cite{Schoeberl:Simon:Zulehner:2010,Takacs:Zulehner:2012}.

In this paper we present a convergence proof for multigrid methods
based on the classical splitting of the analysis
into smoothing property and approximation property, see~\cite{Hackbusch:1985}.

This paper is organized as follows. In Section~\ref{sec:2} we will
introduce the optimality system and discuss its discretization.
In Section~\ref{sec:3} we will introduce an all-at-once multigrid approach.
Its convergence will be proven in Section~\ref{sec:3a}. 
Numerical results which illustrate the convergence result will be presented
in Section~\ref{sec:4}. In Section~\ref{sec:5} we will close with conclusions.

\section{Optimality system and discretization}\label{sec:2}

For setting up the optimality system we need the
space $H^1_0(\Omega)$, the space of functions in $H^1(\Omega)$ vanishing on the boundary.
Moreover, we need the space $L^2_0(\Omega)$, which is the space of functions in $L^2(\Omega)$ with mean value $0$, i.e.,
\begin{equation*}
	L^2_0(\Omega) := \left\{ v \in L^2(\Omega) \;:\; \int_{\Omega} v \;\mbox{d}\xi = 0\right\}.
\end{equation*}
Both spaces are equipped with standard norms, i.e., $\|\cdot\|_{H^1_0(\Omega)}:=\|\cdot\|_{H^1(\Omega)}$ and
$\|\cdot\|_{L^2_0(\Omega)}:=\|\cdot\|_{L^2(\Omega)}$.

The solution of the problem is characterized by the Karush Kuhn Tucker
system (KKT-system), which reads as follows, cf.~\cite{Zulehner:2010} and others.

Find
$(u,p,f,\lambda,\mu)\in [H^1_0(\Omega)]^d\times
L^2_0(\Omega)\times [L^2(\Omega)]^d  \times [H^1_0(\Omega)]^d\times
L^2_0(\Omega) $ such that
\begin{align*}
          (u,\tilde{u})_{L^2(\Omega)} 
		+ (\nabla \lambda, \nabla \tilde{u})_{L^2(\Omega)}
		+ (\mu,\nabla\cdot \tilde{u})_{L^2(\Omega)} 
		&= (u_D,\tilde{u})_{L^2(\Omega)}\\
          (\nabla \cdot \lambda,\tilde{p})_{L^2(\Omega)} 
		&= 0\\
         \alpha (f,\tilde{f})_{L^2(\Omega)} 
		-(\lambda,\tilde{f})_{L^2(\Omega)} 
		&=0\\
          (\nabla u,\nabla\tilde{\lambda})_{L^2(\Omega)} 
		+(p,\nabla\cdot\tilde{\lambda})_{L^2(\Omega)}
		 -(f,\tilde{\lambda})_{L^2(\Omega)} 
		&=0\\
          (\nabla\cdot u,\tilde{\mu})_{L^2(\Omega)} 
		&= 0\
\end{align*}
holds for all $(\tilde{u},\tilde{p},\tilde{f},\tilde{\lambda},\tilde{\mu})\in
 [H^1_0(\Omega)]^d\times L^2_0(\Omega)\times [L^2(\Omega)]^d  \times
[H^1_0(\Omega)]^d\times L^2_0(\Omega)$.

The third line of the KKT-system directly implies $f = \alpha^{-1}
\lambda$.

This allows to eliminate the control $f$, which
leads to the reduced KKT-system, which reads as follows. Find
$x:=(u,p,\lambda,\mu)\in X:=[H^1_0(\Omega)]^d\times L^2_0(\Omega) \times
[H^1_0(\Omega)]^d\times L^2_0(\Omega) $ such that
\begin{equation}\nonumber
        \begin{aligned}
          (u,\tilde{u})_{L^2(\Omega)}
		+ (\nabla \lambda,\nabla \tilde{u})_{L^2(\Omega)} 
		+ (\mu,\nabla \cdot \tilde{u})_{L^2(\Omega)} 
		&= (u_D,\tilde{u})_{L^2(\Omega)}\\
          (\nabla \cdot \lambda,\tilde{p})_{L^2(\Omega)} 
		&= 0\\
          (\nabla u,\nabla \tilde{\lambda})_{L^2(\Omega)} 
		+ (p,\nabla \cdot \tilde{\lambda})_{L^2(\Omega)} 
		- \alpha^{-1} (\lambda,\tilde{\lambda})_{L^2(\Omega)}  
		&=0\\
          (\nabla\cdot u,\tilde{\mu})_{L^2(\Omega)} 
		&= 0\\
        \end{aligned}                
\end{equation}
holds for all $(\tilde{u},\tilde{p},\tilde{\lambda},\tilde{\mu})\in X$.

Certainly, the KKT-system can be rewritten as one variational equation
as follows. Find $x \in X $ such that
\begin{equation}\label{eq:gen:bil}
        \mcB(x,\tilde{x}) = \mcF(\tilde{x}) \qquad \mbox{for all }\tilde{x} \in X,
\end{equation}
where
\begin{align*}\nonumber
      &  \mcB((u,p,\lambda,\mu),(\tilde{u},\tilde{p},\tilde{\lambda},\tilde{\mu})) :=
		(u,\tilde{u})_{L^2(\Omega)} 
		+(\nabla \lambda,\nabla \tilde{u})_{L^2(\Omega)}
		+(\mu,\nabla\cdot\tilde{u})_{L^2(\Omega)}
	\\&\qquad
		+(\nabla \cdot\lambda,\tilde{p})_{L^2(\Omega)}
		+(\nabla u,\nabla \tilde{\lambda})_{L^2(\Omega)}
		+(p,\nabla\cdot\tilde{\lambda})_{L^2(\Omega)}
		- \alpha^{-1}(\lambda,\tilde{\lambda})_{L^2(\Omega)}
	\\&\qquad
		+ (\nabla\cdot u,\tilde{\mu})_{L^2(\Omega)}\mbox{ and}
	\\
	&        \mcF(\tilde{u},\tilde{p},\tilde{\lambda},\tilde{\mu}) := (u_D,\tilde{u})_{L^2(\Omega)}.
\end{align*}

We are interested in finding an approximative solution for equation~\eqref{eq:gen:bil}. Both, the
proposed solution strategy and the convergence analysis,
follow the abstract framework introduced in~\cite{Takacs:Zulehner:2012}. The
conditions, \citecond{(A1)}, \citecond{(A1a)}, \citecond{(A3)} and 
\citecond{(A4)}, mentioned in the present paper are the same conditions 
as in~\cite{Takacs:Zulehner:2012}.

For simplicity, we introduce the following notation.
\begin{notation}
	Throughout this paper,~$C>0$ is a generic constant, independent of
	the grid level~$k$ and the choice of the parameter~$\alpha$.
	For any scalars~$a$ and~$b$, we write~$a \les b$ (or $b \ges a$) if there is a 
	constant~$C>0$ such that~$a < C\,b$. We write~$a\eqsim b$ if~$a\les b\les a$.
\end{notation}

The following property guarantees existence and uniqueness of the solution.
\begin{description}
        \item[\citecond{(A1)}] The relation
                \begin{equation*}
                        \|x\|_{X} \les \nnsup{\tilde{x}}{X}
                                \frac{\mcB(x,\tilde{x})}{\|\tilde{x}\|_{X}}
                        \les \|x\|_{X}
                \end{equation*}
                holds for all $x\in X$.
\end{description}

In~\cite{Zulehner:2010}
it was shown that condition~\citecond{(A1)} is satisfied
for $X:=Y\times Y$, where $Y:=U\times P$, $U:=[H^1_0(\Omega)]^d$, $P:=L^2_0(\Omega)$,
equipped with norms
\begin{equation*}
        \|x\|_X^2  := \|(u,p,\lambda,\mu)\|_X^2 := \|(u,p)\|_Y^2 +
\alpha^{-1} \|(\lambda,\mu)\|_Y^2,
\end{equation*}
where
\begin{align*}
        \|(u,p)\|_Y^2 &:= \|u\|_U^2 + \|p\|_{P}^2,\\
        \|u\|_U^2 &:= \|u\|_{L^2(\Omega)}^2+\alpha^{1/2} \|u\|_{H^1(\Omega)}^2\mbox{ and}\\
         \|p\|_P^2 &:=  \nnsup{w}{[H^1_0(\Omega)]^d} \frac{\alpha (p,\nabla \cdot
w)_{L^2(\Omega)}^2}{\|w\|_{L^2(\Omega)}^2 + \alpha^{1/2}
\|w\|_{H^1(\Omega)}^2}.
\end{align*}

Using the following notation, we can express the norms in a nicer way.
\begin{notation}
	For any Hilbert space~$A$, the symbol $A^*$ denotes its dual space equipped
	with the dual norm
        \begin{equation*}
                \|u\|_{A^*} := \nnsup{w}{A}\frac{\langle u,w \rangle}{\|w\|_A},
        \end{equation*}
	where $\langle u,\cdot\rangle := u(w)$ denotes the duality pairing.

	For any Hilbert space~$A$ and any scalar $a>0$, the symbol~$a\, A$ 
	denotes the space on the underlying set of the Hilbert space~$A$ equipped with the norm
        \begin{equation*}
                \|u\|_{a\,A}^2 := a \|u\|_A^2.
        \end{equation*}
        For any two Hilbert spaces~$A$ and~$B$, the symbol
        $A\cap B$ denotes the space on the intersection of the underlying sets, 
	$\{u\in A \cap B\}$, equipped with the norm
        \begin{equation*}
                \|u\|_{A\cap B}^2 := \|u\|_A^2 + \|u\|_B^2,
	\end{equation*}
	and the symbol~$A+B$ denotes the space on the algebraic sum of the
	underlying sets, $\{u_1+u_2\;:\; u_1\in A, u_2\in B\}$, equipped with the norm
        \begin{equation*}
                \|u\|_{A+B}^2 := \inf_{u_1\in A, u_2\in B, u=u_1+u_2} 
			\|u_1\|_A^2 + \|u_2\|_B^2.
        \end{equation*}
\end{notation}

The spaces $A^*$, $a\, A$, $A\cap B$ and $A+B$ are Hilbert spaces. The fact that $A^*$
is a Hilbert space follows directly from the Riesz representation theorem, see, e.g.,
Theorem~1.2 in~\cite{Adams:Fournier}. The fact that $a\, A$ is a Hilbert space
is obvious and for the latter two see, e.g., Lemma~2.3.1 in~\cite{Bergh:Loefstroem:1976}.

We immediately see, that the norm on $U$ can be rewritten as follows
\begin{equation*}
        \|u\|_U = \|u\|_{L^2(\Omega)\cap \alpha^{1/2} H^1(\Omega)}.
\end{equation*}

To reformulate the norm $\|\cdot\|_P$, we need a regularity assumption.

\begin{description}
        \item[\citecond{(R)}] \emph{Regularity of the generalized Stokes problem.}
        Let $f\in [L^2(\Omega)]^d$ and $g\in H^1_0(\Omega)\cap L^2_0(\Omega)$ be arbitrarily 
        but fixed and
        $(u,p)\in [H^1_0(\Omega)]^d\times L^2_0(\Omega)$ be the solution 
        of the Stokes problem, i.e., such that
        \begin{equation*}
        \begin{array}{lclclcl}
          (\nabla u,\nabla \tilde{u})_{L^2(\Omega)}
          	 &+&(p,\nabla\cdot\tilde{u})_{L^2(\Omega)}  &=&(f,\tilde{u})_{L^2(\Omega)}\\
          (\nabla\cdot u,\tilde{p})_{L^2(\Omega)} &&& =& (g,\tilde{p})_{L^2(\Omega)} \\
        \end{array}        
        \end{equation*}
        holds for all $(\tilde{u},\tilde{p}) \in [H^1_0(\Omega)]^d\times L^2_0(\Omega)$.

        Then $(u,p)\in [H^2(\Omega)]^d\times H^1(\Omega)$ and
        \begin{equation*}
                \|u\|_{H^2(\Omega)}^2 + \|p\|_{H^1(\Omega)}^2 \les \|f\|_{L^2(\Omega)}^2+\|g\|_{H^1(\Omega)}^2.
        \end{equation*}
\end{description}

This condition is satisfied for convex polygonal domains, see Lemma~2.1 in~\cite{Takacs:2013} which
is a direct consequence of Theorem~2 in~\cite{Kellogg:Osborn:1976}.

\begin{lemma}\label{lem:pe}
	If~\citecond{(R)} is satisfied, then
	\begin{equation*}
		\|p\|_P \eqsim \|p\|_{\alpha H^1(\Omega)+\alpha^{1/2} L^2(\Omega)}
	\end{equation*}
	holds for all $p\in L^2_0(\Omega)$.
\end{lemma}

This lemma was shown in Theorem~3.2 in~\cite{Olshanskii:Peters:Reusken:2005}
under a regularity assumption, which is weaker than regularity 
assumption~\citecond{(R)}.

The discretization of problem~\eqref{eq:gen:bil} is done using
standard finite element techniques.
We assume to have a sequence of girds obtained
by uniform refinement. On each
grid level~$k$, we discretize the problem using the Galerkin approach,
i.e., we have finite dimensional
spaces $X_k \subseteq X$ and consider the following problem. Find $x_k
\in X_k $ such that
\begin{equation}\label{eq:galerkin}
        \mcB(x_k,\tilde{x}_k) = \mcF(\tilde{x}_k) \qquad \mbox{for all
}\tilde{x}_k \in X_k.
\end{equation}
Using a nodal basis, we can represent this problem in matrix-vector
notation as follows:
\begin{equation}\label{eq:gen:bil:mv}
                \mcA_k\,\ul{x}_k = \ul{\bff}_k.
\end{equation}
Here and in what follows, any underlined quantity, like~$\ul{x}_k$, is
the representation of the corresponding
non-underlined quantity, here~$x_k$, with respect to a nodal basis of
the corresponding Hilbert space, here~$X_k$.

Existence and uniqueness of the
discretized problem is guaranteed
by the following condition.
\begin{description}
        \item[\citecond{(A1a)}] The relation
                \begin{equation*}
                        \|x_k\|_{X} \les \nnsup{\tilde{x}_k}{X_k}
                                \frac{\mcB(x_k,\tilde{x}_k)}{\|\tilde{x}_k\|_{X}}
                        \les \|x_k\|_{X}
                \end{equation*}
                holds for all $x_k\in X_k$.
\end{description}
Due to the fact that the model problem is indefinite,
condition~\citecond{(A1)}  does not imply condition~\citecond{(A1a)}.
For the Stokes problem itself, it is well-known that such a condition
(also known as \emph{discrete inf-sup condition})
can only be guaranteed if the discretization is chosen appropriately. The same
is true for the Stokes control problem. Fortunately, we can show
the discrete inf-sup condition~\citecond{(A1a)} for
the Stokes control problem based on pre-existing knowledge on the
discrete inf-sup condition
for the Stokes problem. This allows to show that all discretizations which are
suitable for the Stokes flow problem are also suitable
for the Stokes control problem.

We choose the space~$X_k$ as follows:
\begin{equation*}
      X_k := Y_k \times Y_k\qquad \mbox{where} \qquad Y_k:=U_k \times P_k\qquad
\end{equation*}
and the choice of $U_k\subseteq U=[H^1_0(\Omega)]^d$ 
and~$P_k \subseteq P = L^2_0(\Omega)$ is discussed below. Note that $X_k$ has product 
structure and that the state and the
adjoined state (Lagrange multipliers) are
discretized the same way. The same has already been done for
optimal control problems with elliptic state equation, cf.~\cite{Takacs:Zulehner:2012} and
many others.

Due to the fact that the grids are obtained by uniform refinement, the
discrete subsets are nested, i.e.,
$U_{k} \subseteq U_{k+1}$ and $P_{k} \subseteq P_{k+1}$. Therefore,
also $X_{k} \subseteq X_{k+1}$ holds.

The next step is to show condition~\citecond{(A1a)}. We have seen that
the analysis done in~\cite{Zulehner:2010}, applied
to the infinite dimensional spaces, shows condition~\citecond{(A1)}.
If the analysis done
in~\cite{Zulehner:2010} is applied to the discretized spaces, we obtain that
\begin{equation}\label{eq:a1pseudo}
        \|x_k\|_{X_k} \les \nnsup{\tilde{x}_k}{X_k}
                \frac{\mcB(x_k,\tilde{x}_k)}{\|\tilde{x}_k\|_{X_k}}
        \les \|x_k\|_{X_k}
\end{equation}
is satisfied for all $x_k \in X_k$, where
\begin{align*}
        \|x_k\|_{X_k}^2  &:= \|(u_k,p_k,\lambda_k,\mu_k)\|_{X_k}^2 :=
\|(u_k,p_k)\|_{Y_k}^2 + \alpha^{-1}
\|(\lambda_k,\mu_k)\|_{Y_k}^2,\\
        \|(u_k,p_k)\|_{Y_k}^2& := \|u_k\|_{U}^2+ \|p_k\|_{P_k}^2,\\
         \|p_k\|_{P_k}^2 &:=  \nnsup{w_k}{U_k} \frac{\alpha (p_k,\nabla
\cdot w_k)_{L^2(\Omega)}^2}{\|w_k\|_{L^2(\Omega)}^2 +
\alpha^{1/2} \|w_k\|_{H^1(\Omega)}^2}\mbox{ and}
\end{align*}
$\|\cdot\|_{U}^2$ is as above.

Note that this is not condition~\citecond{(A1a)}, as the norms
$\|\cdot\|_P$ and $\|\cdot\|_{P_k}$ are not equal. For
showing condition~\citecond{(A1a)}, it suffices to show that these two
norms are equivalent which implies also the equivalence of the norms $\|\cdot\|_X$ and 
$\|\cdot\|_{X_k}$. This can be shown using the following condition.
\begin{description}
        \item[\citecond{(S)}]
                The discretization of $P$ is $H^1$-conforming, i.e., $P_k\subseteq H^1(\Omega)$, and
		the \emph{weak inf-sup condition}
                \begin{equation*}
                        \nnsup{u_k}{U_k} \frac{(\nabla\cdot u_k,p_k)_{L^2(\Omega)}
}{\|u_k\|_{L^2(\Omega)} } \ges \|\nabla p_k\|_{L^2(\Omega)}
                \end{equation*}
                holds for all $p_k\in P_k$.
\end{description}
\begin{lemma}
        Assume that the discretization satisfies condition~\citecond{(S)}.
	Then condition~\citecond{(A1a)} is satisfied for the model problem.
\end{lemma}
\begin{proof}
	Lemma~2.2 in~\cite{Olshanskii:2012} states (provided that \citecond{(S)}
	is satisfied) that ${\|\cdot\|_P}\eqsim{\|\cdot\|_{P_k}}$ is satisfied.
	A direct consequence is $\|\cdot\|_X\eqsim\|\cdot\|_{X_k}$. Therefore,
	condition~\eqref{eq:a1pseudo} implies condition~\citecond{(A1a)}.
\qed\end{proof}

Note that condition~\citecond{(S)} is a standard condition which ensures that the
chosen discretization is stable for the Stokes problem.
In~\cite{Bercovier:Pironneau:1979,Verfuerth:1984} it was shown that
condition~\citecond{(S)} is satisfied for the Taylor-Hood element 
($P1-P2$-element) for polygonal
domains where at least one vertex of each element is located in the interior
of the domain. Here and in what follows we assume that the problem is
discretized with the Taylor-Hood element and that the mesh satisfies
the named condition.

\section{An all-at-once multigrid method}\label{sec:3}

The problem shall be solved with an all-at-once multigrid method. The
abstract algorithm for solving the discretized
equation~\eqref{eq:gen:bil:mv} on grid level~$k$ reads as follows.
Starting from an initial approximation~$\ul{x}^{(0)}_k$,
one iterate of the multigrid method is given by the following two steps:
\begin{itemize}
        \item \emph{Smoothing procedure:} Compute
              \begin{equation} \nonumber
                   \ul{x}^{(0,m)}_k := \ul{x}^{(0,m-1)}_k + \hat{\mcA}_k^{-1}
                                    \left(\ul{\bff}_k -\mcA_k\;\ul{x}^{(0,m-1)}_k\right)
                                    \qquad \mbox{for } m=1,\ldots,\nu
              \end{equation}
        with $\ul{x}^{(0,0)}_k=\ul{x}^{(0)}_k$. The choice of
        the smoother (or, in other words, of the preconditioning matrix
$\hat{\mcA}_k^{-1}$) will be discussed below. \vspace{.2cm}
        \item \emph{Coarse-grid correction:}
                \begin{itemize}
                     \item Compute the defect 
                        $\ul{r}_k^{(1)} := \ul{\bff}_k -\mcA_k\;\ul{x}^{(0,\nu)}_k$
                        and restrict it to grid level $k-1$ using
                        an restriction matrix $I_k^{k-1}$:
                        \begin{equation}\nonumber
                              \ul{r}_{k-1}^{(1)} := I_k^{k-1} \left(\ul{\bff}_k -\mcA_k
                              \;\ul{x}^{(0,\nu)}_k\right).
                        \end{equation}
                     \item Solve the coarse-grid problem
                        \begin{equation}\label{eq:coarse:grid:problem}
                            \mcA_{k-1} \,\ul{p}_{k-1}^{(1)} =\ul{r}_{k-1}^{(1)}
                        \end{equation}
                        approximatively.
                     \item Prolongate $\ul{p}_{k-1}$  to the
                          grid level $k$ using an prolongation 
                          matrix $I^k_{k-1}$ and add
                          the result to the previous iterate:
                          \begin{equation}\nonumber
                               \ul{x}_{k}^{(1)} := \ul{x}^{(0,\nu)}_k +
                                I_{k-1}^k \, \ul{p}_{k-1}^{(1)}.
                          \end{equation}
                \end{itemize}
\end{itemize}
As we have assumed to have nested spaces, the intergrid-transfer
matrices $I_{k-1}^k$ and $I_k^{k-1}$ can be chosen in
a canonical way:
$I_{k-1}^k$ is the canonical embedding and the restriction $I_k^{k-1}$
is its transpose.

If the problem on the coarser grid is solved exactly (two-grid
method), the coarse-grid correction is given by
\begin{equation} \label{eq:method:cga}
        \ul{x}_k^{(1)} := \ul{x}_k^{(0,\nu)} +
        I_{k-1}^{k} \, \mcA_{k-1}^{-1} \,  I_{k}^{k-1}
        \left( \ul{\bff}_k - \mcA_k \;\ul{x}_k^{(0,\nu)}\right).
\end{equation}
In practice the problem~\eqref{eq:coarse:grid:problem} is
approximatively solved by applying one step (V-cycle)
or two steps (W-cycle) of the multigrid method, recursively. On
the coarsest grid level ($k=0$) the problem~\eqref{eq:coarse:grid:problem} is 
solved exactly.

To construct a multigrid convergence result based on Hackbusch's
splitting of the analysis into smoothing property and approximation
property, we have to introduce an appropriate framework.

Convergence is shown on the spaces $X_k$, which are equipped with an $L^2$-like
norms $\hnorm\cdot\hnorm_{0,k}$, which are defined a follows:
\begin{equation*}
        \hnorm x_k\hnorm_{0,k}^2 :=\|\ul{x}_k\|_{\mcL_k}^2
        :=(\mcL_k \ul{x}_k,\ul{x}_k)_{\ell^2},
\end{equation*}
where
\begin{equation}\label{eq:defMcL}
        \mcL_k:=
        \left(
                \begin{array}{cccc}
                        \varphi_{\alpha,k} M_{U,k}\hspace{-.4cm}\mbox{} \\
                        &\alpha h_k^{-2} \varphi_{\alpha,k}^{-1}M_{P,k}\hspace{-.4cm}\mbox{}\\
                        &&\alpha^{-1} \varphi_{\alpha,k} M_{U,k}\hspace{-.4cm}\mbox{}\\
                        &&& h_k^{-2} \varphi_{\alpha,k}^{-1} M_{P,k}\ \\
                \end{array}
        \right),
\end{equation}
and $\varphi_{\alpha,k}:=1+\alpha^{1/2} h_k^{-2}$ and $M_{U,k}$ and $M_{P,k}$ are the mass matrices,
representing the $L^2$-inner product on $U_k$ and $P_k$, respectively.
Based on the norm $\hnorm \cdot\hnorm_{0,k}$, we can introduce the residual norm $\hnorm \cdot\hnorm_{2,k}$
using
\begin{equation*}
         \hnorm x_k \hnorm_{2,k} :=  \sup_{\tilde{x}_k\in X_k} 
                 \frac{\mcB \left(x_k, \tilde{x}_k\right)}{\hnorm \tilde{x}_k\hnorm_{0,k}}.
\end{equation*}

Smoothing property and approximation property read as follows.
\begin{itemize}
        \item \emph{Smoothing property:}
                \begin{equation} \label{eq:smp}
			\hnorm x_k^{(0,\nu)}-x_k^* \hnorm_{2,k}
                         \le \eta(\nu) \hnorm x_k^{(0)}-x_k^*\hnorm_{0,k}
                \end{equation}
                should hold for some function $\eta(\nu)$
                with $\lim_{\nu\rightarrow\infty}\eta(\nu)= 0$. Here and in what follows,
		$x_k^*\in X_k$ is the exact solution of the discretized problem~\eqref{eq:gen:bil:mv}.
        \item \emph{Approximation property:}
                \begin{equation} \label{eq:apprp}
                        \hnorm x_k^{(1)}-x_k^*\hnorm_{0,k} \le
                        C_A \hnorm x_k^{(0,\nu)}-x_k^*\hnorm_{2,k}
                \end{equation}        
                should hold for some constant $C_A>0$.
\end{itemize}
It is easy to see that, if we combine both conditions, we obtain
\begin{equation}\nonumber
                        \hnorm x_k^{(1)}-x_k^*\hnorm_{0,k} \le
                        q(\nu) \hnorm x_k^{(0)}-x_k^*\hnorm_{0,k},
\end{equation}        
where $q(\nu)=C_A\eta(\nu)$, i.e., that the two-grid method converges
for $\nu$ large enough.
The convergence of the W-cycle multigrid method can be shown under
mild assumptions, see e.g.~\cite{Hackbusch:1985}.

The choice of an appropriate smoother is a key issue in constructing
such a multigrid method.
Here, we introduce one smoother which is appropriate for a
large class of problems including the model problem:
\emph{normal equation smoothers}, cf.~\cite{Brenner:1996}, which read as follows.
\begin{equation}\nonumber
        \ul{x}^{(0,m)}_k := \ul{x}^{(0,m-1)}_k + \tau
                         \underbrace{\mcL_k^{-1} \mcA_k \mcL_k^{-1}}_{\displaystyle
\hat{\mcA}_k^{-1}:=}
                        \left(\ul{\bff}_k -\mcA_k \;\ul{x}^{(0,m-1)}_k\right)
                \qquad \mbox{for } m=1,\ldots,\nu.
\end{equation}
Here, a fixed~$\tau>0$ has to be chosen such that the
spectral radius~$\rho(\tau \hat{\mcA}_k^{-1}\mcA_k)$ is bounded away
from~$2$ on all grid
levels~$k$ and for all choices of the parameter~$\alpha$.

Using a standard inverse inequality, one can show that
\begin{equation*}
		\|x_k\|_{X} \les \hnorm x_k\hnorm_{0,k} 
\end{equation*}
is satisfied for all $x_k \in X_k$. Based on this result, using an eigenvalue 
analysis one can show the following lemma, cf.~\cite{Brenner:1996}.
\begin{lemma}
        The damping parameter $\tau>0$ can be chosen independent
        of grid level~$k$ and the choice of the parameter~$\alpha$ such that
        \begin{equation*}
                \tau\, \rho(\hat{\mcA}_k^{-1}\mcA_k) \le 2-\epsilon < 2,
        \end{equation*}
        holds for some constant $\epsilon>0$. For this choice of~$\tau$, there is a
        constant~$C_S>0$, independent of the grid level~$k$ and the choice of the
	parameter~$\alpha$,
        such that the smoothing property~\eqref{eq:smp} is satisfied with rate
        \begin{equation*}
                \eta(\nu):=C_S \nu^{-1/2}.
        \end{equation*}
\end{lemma}

Certainly, the iteration procedure~\eqref{eq:smp} should be efficient-to-apply.
Using the fact, that the mass matrices $M_{U,k}$ and $M_{P,k}$ in~\eqref{eq:defMcL} and their diagonals are
spectrally equivalent under weak assumptions, for the practical
realization of the smoother these matrices can be replaced by their diagonals.

\section{A convergence proof}\label{sec:3a}

The proof of the approximation property is done using the
approximation theorem introduced in~\cite{Takacs:Zulehner:2012}
which requires besides the conditions~\citecond{(A1)}
and~\citecond{(A1a)} two more conditions (conditions~\citecond{(A3)} and 
\citecond{(A4)}) involving, besides the Hilbert space $X$, two more Hilbert
spaces $X_{-,k}:=(X_-,\|\cdot\|_{X_{-,k}})$ and
 $X_{+,k}:=(X_+,\|\cdot\|_{X_{+,k}})$, which are chosen as follows.

As weaker space, we choose $X_-:= Y_- \times
Y_-$, where $Y_-:=U_-\times P_-$, $U_-:=[L^2(\Omega)]^d$ and $P_-:= [H^1_0(\Omega)\cap L^2_0(\Omega)]^*$,
equipped with norms
\begin{align*}
        \|x\|_{X_{-,k}}^2 &:= \|(u,p,\lambda,\mu)\|_{X_{-,k}}^2 :=
\|(u,p)\|_{Y_{-,k}}^2 + \alpha^{-1} \|(\lambda,\mu)\|_{Y_{-,k}}^2,\\
        \|(u,p)\|_{Y_{-,k}}^2 &:= \|u\|_{U_{-,k}}^2 +\|p\|_{P_{-,k}}^2,\\
        \|u\|_{U_{-,k}}^2 &:=h_k^{-2} \|u\|_{[H^1_0(\Omega)+ \alpha^{-1/2} L^2(\Omega)]^*}^2 
		\mbox{ and}\\
        \|p\|_{P_{-,k}}^2 &:=h_k^{-2} \|p\|_{[\alpha^{-1} L^2_0(\Omega)\cap\alpha^{-1/2} H^1_0(\Omega)]^*}^2
\end{align*}
Note that dual spaces are $(X_-)^*:= (Y_-)^*\times (Y_-)^*$, where $(Y_-)^*=(U_-)^*\times (P_-)^*$, 
$(U_-)^*=[L^2(\Omega)]^d$ and $(P_-)^* = H^1_0(\Omega)\cap L^2_0(\Omega)$, equipped with norms
\begin{align*}
        \|\mcF\|_{(X_{-,k})^*}^2 &:= \|(f,g,\zeta,\chi)\|_{(X_{-,k})^*}^2 :=
\|(f,g)\|_{(Y_{-,k})^*}^2 + \alpha \|(\zeta,\chi)\|_{(Y_{-,k})^*}^2,\\
        \|(f,g)\|_{(Y_{-,k})^*}^2 &:= \|f\|_{(U_{-,k})^*}^2 +\|g\|_{(P_{-,k})^*}^2,\\
        \|f\|_{(U_{-,k})^*}^2 &:=h_k^{2} \|f\|_{H^1_0(\Omega)+ \alpha^{-1/2} L^2(\Omega)}^2 
		\mbox{ and}\\
        \|g\|_{(P_{-,k})^*}^2 &:=h_k^{2} \|g\|_{\alpha^{-1} L^2_0(\Omega)\cap\alpha^{-1/2} H^1_0(\Omega)}^2.
\end{align*}

As stronger space, we choose
$X_{+}:=Y_{+}\times Y_{+}$, where $Y_{+}:=U_{+}\times P_{+}$, $U_{+}:=[H^2(\Omega)\cap H^1_0(\Omega)]^d$ and
$P_{+}:=H^1(\Omega)\cap L^2_0(\Omega)$, equipped with norms
\begin{align*}
        \|x\|_{X_{+,k}}^2 &:= \|(u,p,\lambda,\mu)\|_{X_{+,k}}^2 :=
\|(u,p)\|_{Y_{+,k}}^2 + \alpha^{-1} \|(\lambda,\mu)\|_{Y_{+,k}}^2,\\
        \|(u,p)\|_{Y_{+,k}}^2 &:= \|u\|_{U_{+,k}}^2+ \|p\|_{P_{+,k}}^2,\\
        \|u\|_{U_{+,k}}^2 &:= h_k^2 \|u\|_{H^1(\Omega)\cap \alpha^{1/2} H^2(\Omega)}^2
	\mbox{ and}\\
        \|p\|_{P_{+,k}}^2 &:= h_k^2 \|p\|_{\alpha H^2(\Omega)+\alpha^{1/2} H^1(\Omega)}^2.
\end{align*}
The additional conditions read as follows.
\begin{description}
                \item[\citecond{(A3)}] On all grid levels $k$, the approximation error
                result
                \begin{equation*}
                        \inf_{x_k\in X_k} \|x-x_k\|_X \les \| x \|_{X_{+,k}}\qquad \mbox{for all } x\in X_+
                \end{equation*}
                is satisfied.
                \item[\citecond{(A4)}] For all grid levels $k$, all $\mcF\in (X_-)^*$ the solution
                        $x_{\mcF}\in X$ of the problem, 
                        \begin{equation}\label{a4:problem}
                               \mbox{find $x\in X$ such that} \qquad \mcB(x,\tilde{x}) = \mcF(\tilde{x}) \qquad\mbox{ for all } \tilde{x}\in X,
                        \end{equation}
                        satisfies $x_{\mcF}\in X_+$ and the inequality
                        \begin{equation}\label{eq:a4}
                                \|x_{\mcF}\|_{X_{+,k}} \les \|\mcF\|_{(X_{-,k})^*}.
                        \end{equation}
\end{description}

Based on these assumptions, the following theorem shows the approximation property.
\begin{theorem}\label{thrm:main}
        Let for $k=0,1,2,\ldots$ the symmetric matrices $\mcA_k$ be obtained by discretizing 
	problem~\eqref{eq:galerkin} using a sequence of finite-dimensional nested
        subspaces $X_{k-1}\subseteq X_k\subset X$.
        Assume that there are Hilbert spaces $X_+\subseteq X\subseteq X_-$
        with mesh-dependent norms ${\|\cdot\|_{X_{+,k}}}$, ${\|\cdot\|_{X}}$
        and ${\|\cdot\|_{X_{-,k}}}$ such that the conditions~\citecond{(A1)},
	\citecond{(A1a)}, \citecond{(A3)} and~\citecond{(A4)} are satisfied.
        Then the coarse-grid correction~\eqref{eq:method:cga} satisfies the approximation 
				property
	      \begin{equation} \label{eq:apprp:thrm1}
                        \| x_k^{(1)}-x_k^*\|_{X_{-,k}} \le
                        C_A \sup_{\tilde{x}_k\in X_k} 
                        \frac{\mcB\left(
				x_k^{(0,\nu)}-x_k^*,\tilde{x}_k\right)}
				{\|\tilde{x}_k\|_{X_{-,k}}},
        \end{equation}     
	where the constant $C_A$ only depends on the constants 
	that appear (implicitly) in the named conditions.
\end{theorem}

For a proof, see~\cite{Takacs:Zulehner:2012}, Theorem~4.1.

\begin{theorem}
        Condition~\citecond{(A3)} is satisfied.
\end{theorem}
\begin{proof}
	This proof is analogous to the proof of Theorem~4.2 in~\cite{Takacs:2013}. However, to keep
	this paper as self-contained as possible, we give a proof of this theorem.

  Note that it suffices
	to show approximation error results for the individual variables separately.
	Using a standard interpolation operator $\Pi_k:[H^2(\Omega)]^d\rightarrow U_k$, 
	we obtain for the velocity field~$u$
        \begin{equation*}
		\|u - \Pi_k u\|_{L^2(\Omega)}^2 \les h_k^2 \|u \|_{H^1(\Omega)}^2
		\quad\mbox{and}\quad
		\|u - \Pi_k u\|_{H^1(\Omega)}^2 \les h_k^2 \|u \|_{H^2(\Omega)}^2,
        \end{equation*}
	for all $u\in [H^2(\Omega)]^d$ and therefore
        \begin{align*}
		& \inf_{u_k\in U_k} \|u-u_k\|_{U}^2 \le \|u-\Pi_ku\|_{U}^2
                =\|u-\Pi_ku\|_{L^2(\Omega)}^2+\alpha^{1/2} \|u-\Pi_ku\|_{H^1(\Omega)}^2
		\\&\qquad \les h_k^2 \left(\|u\|_{H^1(\Omega)}^2+ \alpha^{1/2}\|u\|_{H^2(\Omega)}^2\right)
                =  \|u\|_{U_{+,k}}^2.
        \end{align*}
	The same can be done for the adjoined velocity~$\lambda$.
	Also for the pressure distribution~$p$ we can do a similar estimate. The estimates
        \begin{equation*}
		\inf_{p_k\in P_k} \|p - p_k\|_{L^2(\Omega)}^2 \les h_k^2 \|p\|_{H^1(\Omega)}^2
		\quad \mbox{and} \quad
		\inf_{p_k\in P_k} \|p - p_k\|_{H^1(\Omega)}^2 \les h_k^2 \|p\|_{H^2(\Omega)}^2
        \end{equation*}
	are standard approximation error results which imply
        \begin{align*}
                &\inf_{p_k\in P_k} \|p-p_k\|_P^2 = \inf_{p_k\in P_k} \|p-p_k\|_{\alpha H^1(\Omega)+\alpha^{1/2} L^2(\Omega)}^2 \\
           &\quad=
			\inf_{\substack{p_k\in P_k\\q_1\in H^1(\Omega)\\q_2\in L^2(\Omega)\\q_1+q_2=p-p_k}} \|q_1\|_{\alpha H^1(\Omega)}^2
+\|q_2\|_{\alpha^{1/2} L^2(\Omega)}^2\\
			&\quad=
			\inf_{\substack{p_1\in H^1(\Omega)\\p_2\in L^2(\Omega)\\p_1+p_2=p}}
			\inf_{p_{1,k}\in P_k} \|p_1-p_{1,k}\|_{\alpha H^1(\Omega)}^2
			+\inf_{p_{2,k}\in P_k} \|p_2-p_{2,k}\|_{\alpha^{1/2} L^2(\Omega)}^2\\
			&\quad\le
			\inf_{\substack{p_1\in H^2(\Omega)\\p_2\in H^1(\Omega)\\p_1+p_2=p}}
			\inf_{p_{1,k}\in P_k} \|p_1-p_{1,k}\|_{\alpha H^1(\Omega)}^2
			+\inf_{p_{2,k}\in P_k} \|p_2-p_{2,k}\|_{\alpha^{1/2} L^2(\Omega)}^2\\
			&\quad\les h_k^2
			\inf_{\substack{p_1\in H^2(\Omega)\\p_2\in H^1(\Omega)\\p_1+p_2=p}}
			\|p_1\|_{\alpha H^2(\Omega)}^2
			+ \|p_2\|_{\alpha^{1/2} H^1(\Omega)}^2
		= h_k^2 \|p\|_{\alpha H^2(\Omega)+\alpha^{1/2} H^1(\Omega)}^2.
        \end{align*}
	The same can be done for the adjoined pressure~$\mu$. This finishes the proof.
\qed\end{proof}

For showing~\citecond{(A4)}, we recall Theorem~4.6 
in~\cite{Takacs:2013} on the regularity of the generalized Stokes problem. For this
purpose, we need a regularity assumption for the
Poisson problem with homogeneous Neumann boundary conditions.
\begin{description}
        \item[\citecond{(R1)}] \emph{Regularity of the Poisson problem.}
        Let $g\in L^2(\Omega)$ and $p\in H^1(\Omega)\cap L^2_0(\Omega)$ be such that
        \begin{equation*}
          (\nabla p,\nabla \tilde{p})_{H^1(\Omega)} = (g,\tilde{p})_{L^2(\Omega)} \mbox{ for all }\tilde{p} \in H^1(\Omega)\cap L^2_0(\Omega).
        \end{equation*}
        Then $p\in H^2(\Omega)$ and
        $
                \|p\|_{H^2(\Omega)} \les \|g\|_{L^2(\Omega)}
        $.
\end{description}
Such a regularity assumption can be guaranteed for convex polygonal domains (see, e.g.,~\cite{Dauge:1988}).

Theorem~4.6 in~\cite{Takacs:2013} directly implies the following theorem.
\begin{theorem}\label{thrm:td:stokes}
	Suppose that the regularity assumptions~\citecond{(R)} and~\citecond{(R1)} are satisfied.
	Let $f \in [L^2(\Omega)]^d$ and $g\in H^1_0(\Omega)\cap L^2_0(\Omega)$. 
	The solution of the problem, find $(u,p) \in Y$ such that
	\begin{align*}
		\alpha^{-1/2} ( u,\tilde{u})_{L^2(\Omega)} 
	    + (\nabla u,\nabla\tilde{u})_{L^2(\Omega)} 
			+(p,\nabla\cdot\tilde{u})_{L^2(\Omega)}
			&=(f,\tilde{u})_{L^2(\Omega)}\\
	          (\nabla\cdot u,\tilde{p})_{L^2(\Omega)} 
			&= (g,\tilde{p})_{L^2(\Omega)}\
	\end{align*}
	for all $(\tilde{u},\tilde{p}) \in Y$, satisfies $(u,p)\in Y_{+}$ and the inequality
	\begin{equation*}
		\begin{aligned}
			&\|u\|_{\alpha^{-1/2} H^1(\Omega)\cap H^2(\Omega) }^2
			+ \|p\|_{\alpha^{1/2} H^2(\Omega)+H^1(\Omega)}^2\\
			&\qquad\les
			\|f\|_{\alpha^{1/2} H^1_0(\Omega)+ L^2(\Omega) }^2
			+ \|g\|_{\alpha^{-1/2} L^2_0(\Omega)\cap H^1_0(\Omega)}^2
		\end{aligned}
	\end{equation*}
	is satisfied.
\end{theorem}
\begin{proof}
	We choose the parameter $\beta$ (which occurs in~\cite{Takacs:2013})
	to be $\beta:=\alpha^{-1/2}$.
\qed\end{proof}

\begin{lemma}\label{lem:step2}
        Suppose that assumptions~\citecond{(R)} and~\citecond{(R1)} are satisfied. Let
        $\mcF\in (X_-)^*$ be arbitrarily but fixed.
	Then, $x_{\mcF}$, the solution of~\eqref{a4:problem},
	satisfies $x_{\mcF}\in X_{+}$ and the bound
	\begin{equation}\label{eq:aux1}
		\|x_{\mcF} \|_{X_{+,k}}^2 \les\|\mcF\|_{(X_{-,k})^*}^2
		+ h_k^2 \left( \|u_{\mcF} \|_{ H^1(\Omega) }^2 + \alpha^{-1} \|\lambda_{\mcF}\|_{ H^1(\Omega) }^2\right).
	\end{equation}
\end{lemma}
\begin{proof}
	Let $\mcF(\tilde{u},\tilde{p},\tilde{\lambda},\tilde{\mu}):=
		(f,\tilde{u})_{L^2(\Omega)} + (g,\tilde{p})_{L^2(\Omega)}+
		(\zeta,\tilde{\lambda})_{L^2(\Omega)} + (\chi,\tilde{\mu})_{L^2(\Omega)}$, where
		$f, \zeta \in [L^2(\Omega)]^d$ and $g,\chi \in H^1_0(\Omega)\cap L^2_0(\Omega)$.
		
	Let $\hat{f}:= f-u_{\mcF}+ \alpha^{-1/2}\lambda_{\mcF}$ and 
	$\hat{\zeta}:=\zeta+\alpha^{-1} \lambda_{\mcF}+\alpha^{-1/2} u_{\mcF}$.
	Then we can rewrite the KKT-system as follows:
        \begin{align*}
              (\nabla \lambda_{\mcF},\nabla \tilde{u})_{L^2(\Omega)} + \alpha^{-1/2}
		( \lambda_{\mcF}, \tilde{u})_{L^2(\Omega)}   +
              (\mu_{\mcF},\nabla \cdot \tilde{u})_{L^2(\Omega)} &=
               (\hat{f}, \tilde{u})_{L^2(\Omega)} \\
              (\nabla \cdot \lambda_{\mcF},\tilde{p})_{L^2(\Omega)}
              & = (g,\tilde{p})_{L^2(\Omega)}
        \end{align*}
	and
        \begin{align*}
              (\nabla u_{\mcF},\nabla \tilde{\lambda})_{L^2(\Omega)}
		+\alpha^{-1/2} ( u_{\mcF}, \tilde{\lambda})_{L^2(\Omega)}
               +(p_{\mcF},\nabla\cdot \tilde{\lambda})_{L^2(\Omega)}
               &=(\hat{\zeta}, \tilde{\lambda})_{L^2(\Omega)}\\
               (\nabla \cdot u_{\mcF},\tilde{\mu})_{L^2(\Omega)} 
               &= (\chi,\tilde{p})_{L^2(\Omega)}.
        \end{align*}
  As $\hat{f}  \in [L^2(\Omega)]^d$, 
	$g\in H^1_0(\Omega)\cap L^2_0(\Omega)$,
  $\hat{\zeta} \in [L^2(\Omega)]^d$
  and $\chi \in H^1_0(\Omega)\cap L^2_0(\Omega)$,
  we obtain using Theorem~\ref{thrm:td:stokes} that $x_{\mcF} \in X_+$ and the following bounds are satisfied:
  \begin{align*}
	&\|\lambda_{\mcF}\|_{\alpha^{-1/2} H^1(\Omega) \cap H^2(\Omega) }^2 +
		\|\mu\|_{\alpha^{1/2} H^2(\Omega) + H^1(\Omega)}^2 \\
	&\qquad \les \|f-u_{\mcF}+\alpha^{-1/2}\lambda_{\mcF} \|_{\alpha^{1/2} H^1_0(\Omega) + L^2(\Omega) }^2 +
		\|g\|_{ \alpha^{-1/2} L^2_0(\Omega) \cap H^1_0(\Omega) }^2 
  \end{align*}
  and
  \begin{align*}
	&\|u_{\mcF}\|_{\alpha^{-1/2} H^1(\Omega) \cap H^2(\Omega) }^2 +
		\|p_{\mcF}\|_{ \alpha^{1/2} H^2(\Omega) + H^1(\Omega) }^2 \\
	&\qquad \les \|\zeta-\alpha^{-1}\lambda_{\mcF}+\alpha^{-1/2}u_{\mcF} \|_{\alpha^{1/2} H^1_0(\Omega) + L^2(\Omega) }^2 +
		\|\chi\|_{ \alpha^{-1/2} L^2_0(\Omega) \cap H^1_0(\Omega) }^2.
  \end{align*}
  We can combine these two estimates and obtain
  \begin{align*}
	&\|u_{\mcF}\|_{ H^1(\Omega) \cap \alpha^{1/2} H^2(\Omega) }^2 
		+ \|p_{\mcF}\|_{  \alpha H^2(\Omega) + \alpha^{1/2} H^1(\Omega) }^2 
	\\&\qquad
		+ \alpha^{-1} \|\lambda_{\mcF}\|_{ H^1(\Omega) \cap \alpha{1/2} H^2(\Omega) }^2 
		+ \alpha^{-1} \|\mu\|_{ \alpha H^2(\Omega) + \alpha^{1/2} H^1(\Omega) }^2
	\\&\qquad\les 
		  \|f\|_{ H^1_0(\Omega) + \alpha^{-1/2} L^2(\Omega) }^2
		+ \|g\|_{ \alpha^{-1} L^2_0(\Omega) \alpha^{-1/2} \cap H^1_0(\Omega) }^2
	\\&\qquad\qquad	
		+ \alpha \|\zeta\|_{ H^1_0(\Omega) + \alpha^{-1/2} L^2(\Omega) }^2 
		+ \alpha \|\chi\|_{ \alpha^{-1} L^2_0(\Omega) \cap \alpha^{-1/2} H^1_0(\Omega) }^2
	\\&\qquad\qquad
		+ \|u_{\mcF} \|_{ H^1_0(\Omega) + \alpha^{-1/2} L^2(\Omega) }^2 
		+ \alpha^{-1} \|\lambda_{\mcF}\|_{ H^1_0(\Omega) + \alpha^{-1/2} L^2(\Omega) }^2.
  \end{align*}
  Note that $\|u_{\mcF} \|_{ H^1_0(\Omega) + \alpha^{-1/2} L^2(\Omega) } \le \|u_{\mcF} \|_{ H^1_0(\Omega)}=\|u_{\mcF} \|_{ H^1(\Omega)}$
  holds because of $u_{\mcF}\in [H^1_0(\Omega)]^d$. As the analogous holds also for $\lambda_{\mcF}$,
  this finishes the proof.
\qed\end{proof}

To show condition~\citecond{(A4)}, we have to bound
$\|u_{\mcF} \|_{H^1(\Omega)}^2 + \alpha^{-1} \|\lambda_{\mcF}\|_{H^1(\Omega)}^2$ from above.
For showing such a result, we need some notation.

As $H^1_0(\Omega)$ is dense in $L^2(\Omega)$, for $u\in [H^2(\Omega)]^d$ the function $-\Delta u \in [L^2(\Omega)]^d$
can be approximated by some function $w^{\epsilon} \in [H^1_0(\Omega)]^d$ such that
\begin{equation*}
	\|-\Delta u - w^{\epsilon}\|_{L^2(\Omega)}^2 \le \epsilon.
\end{equation*}
So, we can introduce an operator $-\Delta^{\epsilon}:[H^2(\Omega)]^d\rightarrow [H^1_0(\Omega)]^d$ such that
\begin{equation*}
	\|-\Delta u - (-\Delta^{\epsilon}) u \|_{L^2(\Omega)}^2 \le \epsilon.
\end{equation*}

Analogously, we  introduce the operator $\nabla^{\epsilon}:H^1(\Omega)
\rightarrow [H^1_0(\Omega)]^d$ such that
\begin{equation*}
	\|\nabla p - \nabla^{\epsilon}p \|_{L^2(\Omega)}^2 \le \epsilon.
\end{equation*}

\begin{lemma}\label{lem:aux2}
	Let $\mcF \in (X_-)^*$ and let
	$x_{\mcF}=(u_{\mcF},p_{\mcF},\lambda_{\mcF},\mu_{\mcF})$ 
	be the solution of~\eqref{a4:problem}. Then $x_{\mcF}$ satisfies the estimate
	\begin{equation}\label{eq:aux2}
		h_k^2\left(\|u_{\mcF}\|_{H^1(\Omega)}^2 + \alpha^{-1}  \|\lambda_{\mcF}\|_{H^1(\Omega)}^2\right)
				\les \|\mcF\|_{(X_{-,k})^*} \|x_{\mcF}\|_{X_{+,k}}.
	\end{equation}
\end{lemma}
\begin{proof}
		Let $\mcF(\tilde{u},\tilde{p},\tilde{\lambda},\tilde{\mu}):=
		(f,\tilde{u})_{L^2(\Omega)} + (g,\tilde{p})_{L^2(\Omega)}+
		(\zeta,\tilde{\lambda})_{L^2(\Omega)} + (\chi,\tilde{\mu})_{L^2(\Omega)}$, where
		$f, \zeta \in [L^2(\Omega)]^d$ and $g,\chi \in H^1_0(\Omega)\cap L^2_0(\Omega)$.

	The idea of this proof is to show that for all $\epsilon>0$ there
	is some $\tilde{x}^{\epsilon}\in X$ such that
	\begin{align}\nonumber
		&\mcF(\tilde{x}^{\epsilon})-\mcB(x_{\mcF},\tilde{x}^{\epsilon})\\
				&\quad\les h_k^{-2} \|\mcF\|_{(X_{-,k})^*} 
					\|x_{\mcF}\|_{X_{+,k}} -\|u_{\mcF}\|_{H^1(\Omega)}^2 
					- \alpha^{-1}  \|\lambda_{\mcF}\|_{H^1(\Omega)}^2 \label{eq:aux2a}\\
			&\quad\quad+ \epsilon (\alpha^{1/2} + \alpha^{-1/2}) h_k^{-1}
				(\|x_{\mcF}\|_{X_{+,k}} + \|\mcF\|_{(X_{-,k})^*})+\epsilon^2.\nonumber
	\end{align}
	Note that the left-hand-side of the inequality is $0$. Therefore, 
	this would be sufficient to show the statement of the lemma, as
	$\epsilon>0$ can be chosen arbitrarily small.

	In the following, we show that~\eqref{eq:aux2a} is satisfied for the choice
	$\tilde{x}^{\epsilon}:=(-\Delta^{\epsilon} u_{\mcF}, $ $
	-\nabla\cdot\nabla^{\epsilon} p_{\mcF},
	\Delta^{\epsilon} \lambda_{\mcF},
	\nabla\cdot\nabla^{\epsilon} \mu_{\mcF})$.
	We estimate the individual summands of $\mcF(\tilde{x}^{\epsilon})-
	\mcB(x_{\mcF},\tilde{x}^{\epsilon})$ separately. For the first one, we obtain
	\begin{align*}
		& -(u_{\mcF},-\Delta^{\epsilon} u_{\mcF})_{L^2(\Omega)}	
			\le- (u_{\mcF},-\Delta u_{\mcF})_{L^2(\Omega)} + \epsilon \|u_{\mcF}\|_{L^2(\Omega)}\\
			&\quad= - (\nabla u_{\mcF},\nabla u_{\mcF})_{L^2(\Omega)} + \epsilon \|u_{\mcF}\|_{L^2(\Omega)}
			\les - \|u_{\mcF}\|_{H^1(\Omega)}^2 + \epsilon h_k^{-1} \|x_{\mcF}\|_{X_{+,k}}
	\end{align*}
	due to the fact that $u_{\mcF}\in [H^2(\Omega)\cap H^1_0(\Omega)]^d$ and
	due to Friedrichs' inequality. The same can be done for
	$\alpha^{-1} (\lambda_{\mcF},\Delta^{\epsilon} \lambda_{\mcF})_{L^2(\Omega)}$.
	
	For the next two summands,
	\begin{align*}
		&-(\nabla u_{\mcF},\nabla \Delta^{\epsilon} \lambda_{\mcF})_{L^2(\Omega)} - (\nabla \lambda_{\mcF},\nabla (-\Delta^{\epsilon}) u_{\mcF})_{L^2(\Omega)}\\
		&\quad = (\Delta u_{\mcF}, \Delta^{\epsilon} \lambda_{\mcF})_{L^2(\Omega)} - (\Delta \lambda_{\mcF}, \Delta^{\epsilon} u_{\mcF})_{L^2(\Omega)}\\
		&\quad \le (\Delta u_{\mcF}, \Delta \lambda_{\mcF})_{L^2(\Omega)} - (\Delta \lambda_{\mcF}, \Delta u_{\mcF})_{L^2(\Omega)} + \epsilon(\|\Delta u_{\mcF}\|_{L^2(\Omega)}+\|\Delta \lambda_{\mcF}\|_{L^2(\Omega)}) \\
		&\quad \le \epsilon(\| u_{\mcF} \|_{H^2(\Omega)}+\| \lambda_{\mcF} \|_{H^2(\Omega)})
			\le \epsilon h_k^{-1} (\alpha^{1/4} + \alpha^{-1/4}) \|x_{\mcF}\|_{X_{+,k}}
	\end{align*}
	is satisfied due to the fact that $\Delta^{\epsilon}$ maps into $[H^1_0(\Omega)]^d$.

	For the next two summands, we obtain
	\begin{equation}\nonumber
	\begin{aligned}
		&-(\nabla\cdot u_{\mcF}, \nabla\cdot\nabla^{\epsilon} \mu_{\mcF})_{L^2(\Omega)} - (\nabla\cdot(-\Delta^{\epsilon}) u_{\mcF},\mu_{\mcF})_{L^2(\Omega)}\\  
			&\quad= (\nabla \nabla\cdot u_{\mcF},\nabla^{\epsilon}\mu_{\mcF})_{L^2(\Omega)}- (\Delta^{\epsilon}u_{\mcF},\nabla \mu_{\mcF})_{L^2(\Omega)} \\
			&\quad\le (\nabla \nabla\cdot u_{\mcF},\nabla^{\epsilon}\mu_{\mcF})_{L^2(\Omega)} - (\Delta^{\epsilon}u_{\mcF},\nabla^{\epsilon} \mu_{\mcF})_{L^2(\Omega)} + \epsilon \|\Delta^{\epsilon} u_{\mcF}\|_{L^2(\Omega)} \\
			&\quad\le -(\nabla u_{\mcF},\nabla\nabla^{\epsilon}\mu_{\mcF})_{L^2} - (\Delta u_{\mcF},\nabla^{\epsilon} \mu_{\mcF})_{L^2}+ \epsilon (\|\Delta u_{\mcF}\|_{L^2} +\|\nabla^{\epsilon} p_{\mcF}\|_{L^2}+\epsilon) \\
			&\quad= -(\nabla u_{\mcF},\nabla\nabla^{\epsilon}\mu_{\mcF})_{L^2} + (\nabla u_{\mcF},\nabla \nabla^{\epsilon} \mu_{\mcF})_{L^2} + \epsilon (\|\Delta u_{\mcF}\|_{L^2} +\|\nabla^{\epsilon} p_{\mcF}\|_{L^2}+\epsilon) \\
			&\quad\le  \epsilon (\|u_{\mcF}\|_{H^2(\Omega)} +\|p_{\mcF}\|_{H^1(\Omega)}+2\epsilon)
			\les \epsilon h_k^{-1} (\alpha^{-1/4} + \alpha^{-1/2}) \|x_{\mcF}\|_{X_{+,k}} + \epsilon^2.
	\end{aligned}
	\end{equation}
	The same can be done for $-(\nabla\cdot \lambda_{\mcF},- \nabla\cdot\nabla^{\epsilon} p_{\mcF})_{L^2(\Omega)} - (\nabla\cdot\Delta^{\epsilon} \lambda_{\mcF},p_{\mcF})_{L^2(\Omega)}$.
		
	Let $f_2\in [H^1_0(\Omega)]^d$ and $f_1:=f-f_2$. Then
	\begin{equation*}
		(f_1,-\Delta^{\epsilon} u_{\mcF})_{L^2(\Omega)}
			\les
			\| f_1\|_{\alpha^{-1/2} L^2(\Omega)} \|u_{\mcF}\|_{\alpha^{1/2} 
			H^2(\Omega)} + \epsilon\alpha^{1/4} \| f_1\|_{\alpha^{-1/2} L^2(\Omega)}
	\end{equation*}
	holds as well as
	\begin{align*}
		(f_2,-\Delta^{\epsilon} u_{\mcF})_{L^2(\Omega)}
		&\les (\nabla f_2,\nabla u_{\mcF})_{L^2(\Omega)}+ \epsilon\| f_2\|_{L^2(\Omega)}\\
			&\les
			\| f_2\|_{H^1(\Omega)} \|u_{\mcF}\|_{H^1(\Omega)} + \epsilon\| f_2\|_{H^1(\Omega)}.
	\end{align*}
	This implies
	\begin{align*}\label{eq:aux2c}
		&(f,-\Delta^{\epsilon} u)_{L^2(\Omega)}\\
		&\qquad \les \|f\|_{H^1_0(\Omega)+ \alpha^{-1/2} L^2(\Omega)} 
						\|u\|_{H^1(\Omega)\cap \alpha^{1/2} H^2(\Omega)} \\
		&\qquad\qquad+ \epsilon (1+\alpha^{1/4})
							\|f\|_{H^1_0(\Omega)+ \alpha^{-1/2} L^2(\Omega)}\\
		&\qquad\les \|f\|_{H^1_0(\Omega)+ \alpha^{-1/2} L^2(\Omega)} 
						\|u\|_{H^1(\Omega)\cap \alpha^{1/2} H^2(\Omega)} + \epsilon h_k^{-1}
							(1+\alpha^{1/4})\|\mcF\|_{(X_{-,k})^*}.
	\end{align*}
	Let $p_2\in H^2(\Omega)$ and $p_1:=p_{\mcF}-p_2\in H^1(\Omega)$.
	We have
	\begin{align*}
		(g,-\nabla\cdot\nabla^{\epsilon} p_1)_{L^2(\Omega)} &=
			(\nabla g,\nabla^{\epsilon} p_1)_{L^2(\Omega)}\\
			&\les \|g\|_{ \alpha^{-1/2} H^1(\Omega) }
					\|p_1\|_{ \alpha^{1/2} H^1(\Omega) } + \epsilon \|g\|_{H^1(\Omega)}.
	\end{align*}
	Moreover, using $g\in H^1_0(\Omega)$, we have also
	\begin{align*}
		&(g,-\nabla\cdot\nabla^{\epsilon} p_2)_{L^2(\Omega)} 
		=	(\nabla g,\nabla^{\epsilon} p_2)_{L^2(\Omega)} 
		\les	(\nabla g,\nabla p_2)_{L^2(\Omega)} + \epsilon \|g\|_{H^1(\Omega)}\\
		&\quad =-( g,\nabla \cdot \nabla p_2)_{L^2(\Omega)} + \epsilon \|g\|_{H^1(\Omega)}
			\le \|g\|_{ \alpha^{-1} L^2(\Omega) }
					\|p_2\|_{ \alpha H^2(\Omega) } + \epsilon \|g\|_{H^1(\Omega)}
	\end{align*}
	and therefore	
	\begin{align*}
		&(g,-\nabla\cdot\nabla^{\epsilon} p_{\mcF})_{L^2(\Omega)}\\
			& \qquad \les \| g\|_{\alpha^{-1} L^2(\Omega)\cap \alpha^{-1/2} H^1(\Omega) }
			\|p_{\mcF}\|_{\alpha^{1/2} H^1(\Omega) + \alpha H^2(\Omega)} + \epsilon \|g\|_{H^1(\Omega)}\\
			& \qquad \les \| g\|_{\alpha^{-1} L^2(\Omega)\cap \alpha^{-1/2} H^1(\Omega) }
			\|p_{\mcF}\|_{\alpha^{1/2} H^1(\Omega) + \alpha H^2(\Omega)} + \epsilon h_k^{-1}\alpha^{1/4}
							\|\mcF\|_{(X_{-,k})^*}
	\end{align*}
	is satisfied. The same can be done for $(\zeta,\Delta^{\epsilon}\lambda_{\mcF})_{L^2(\Omega)}$ and 
	$(\chi,\nabla\cdot\nabla^{\epsilon}\mu_{\mcF})_{L^2(\Omega)}$.
	
	Combining these results, we immediately obtain~\eqref{eq:aux2a}, which
	finishes the proof.
\qed\end{proof}

\begin{theorem}
	Condition~\citecond{(A4)} is satisfied.
\end{theorem}
\begin{proof}
	By combining~\eqref{eq:aux1} and~\eqref{eq:aux2}, we obtain 
	\begin{align*}
		\|x_{\mcF}\|_{X_{+,k}}^2 \le C \left( \|\mcF\|_{(X_{-,k})^*}^2 + 
			\|\mcF\|_{(X_{-,k})^*} \|x_{\mcF}\|_{X_{+,k}} \right)
	\end{align*}
	for some constant $C>0$ (independent of $k$ and $\beta$) which implies
	\begin{align*}
		\|x_{\mcF}\|_{X_{+,k}} \le \frac12 \left(C+\sqrt{4C+C^2}\right) \|\mcF\|_{(X_{-,k})^*},
	\end{align*}
	i.e.~\eqref{eq:a4}, which finishes the proof.
\qed\end{proof}

So, we have shown condition~\citecond{(A4)}. So, Theorem~\ref{thrm:main}
implies the approximation property.
Note, that we have now shown the approximation property in the 
norm $\|\cdot\|_{X_{-,k}}$, i.e.,~\eqref{eq:apprp:thrm1}. The
next step is to show the approximation property 
in the norm-pair $\hnorm \cdot\hnorm_{0,k}$ and $\hnorm \cdot\hnorm_{2,k}$, i.e.,~\eqref{eq:apprp}.

To show~\eqref{eq:apprp}, the following lemma is sufficient.
\begin{lemma}\label{lem:equiv}
The inequality
\begin{equation}\label{eq:equiv}
	\hnorm x_k\hnorm_{0,k} \les \|x_k\|_{X_{-,k}} 
\end{equation}
is satisfied for all $x_k \in X_k$.
\end{lemma}
\begin{proof}
	The proof of this lemma is based on Lemma~4.7 in~\cite{Takacs:2013}.
	Lemma~4.7 states (in the notation of the present paper and for
	the choice $\beta:=\alpha^{-1/2}$) that
	\begin{equation}\nonumber
		\alpha^{-1/2} \hnorm y_k\hnorm_{Y,0,k}^2 \les \alpha^{-1/2} \|y_k\|_{Y_{-,k}}^2
	\end{equation}
	is satisfied for all $y_k \in Y_k$. As for $x_k=(y_k,\psi_k)\in X_k=Y_k \times Y_k$ both,
	\begin{equation}\nonumber
		\hnorm (y_k,\psi_k) \hnorm_{0,k}^2 = \hnorm y_k \hnorm_{Y,0,k}^2
					+\alpha^{-1} \hnorm \psi_k \hnorm_{Y,0,k}^2
	\end{equation}
	and
	\begin{equation}\nonumber
		\| (y_k,\psi_k) \|_{X_{-,k}}^2 = \| y_k \|_{Y_{-,k}}^2
					+\alpha^{-1} \| \psi_k \|_{Y_{-,k}}^2,
	\end{equation}
	is satisfied by definition, \eqref{eq:equiv} follows immediately.
\qed\end{proof}

So, we have shown the approximation property~\eqref{eq:apprp}. So, we obtain the following
overall convergence result.
\begin{theorem}
        Assume that
        \begin{itemize}
                \item the regularity assumptions~\citecond{(R)} and~\citecond{(R1)} are
				satisfied on the domain~$\Omega$,
                \item the problem is discretized using the Taylor-Hood element and
                \item the normal equation smoother introduced above is used as smoother.
        \end{itemize}

        Then the two-grid method converges if sufficiently many smoothing
steps are applied, i.e., we have
        \begin{equation*}
                \hnorm x_k^{(1)}-x_k^*\hnorm_{0,k} \le q(\nu) \hnorm x_k^{(0)}-x_k^*\hnorm_{0,k},
        \end{equation*}
        with $q(\nu):=C_S\, C_A\, \nu^{-1/2}$,
        where the constants~$C_A$ and~$C_S$ are independent of the grid
				level~$k$ and the choice of the parameter~$\alpha$.
\end{theorem}

Note that we have shown that the method converges in the norm $\hnorm \cdot \hnorm_{0,k}$ if sufficiently
many pre-smoothing steps are applied. The application of post-smoothing steps does not derogate
the convergence because the proposed smoother is power-bounded. Moreover, if we assume that only post-smoothing steps
are applied, the combination of smoothing property and approximation property (which now have to be combined
in the inverse order) leads to convergence in the residual norm $\hnorm \cdot \hnorm_{2,k}$. Again, due
to power-boundedness of the smoother, the method stays convergent if, besides sufficiently many post-smoothing
steps, also pre-smoothing steps are applied.

For all the mentioned cases, the convergence of the W-cycle multigrid method follows under weak 
assumptions, cf.~\cite{Hackbusch:1985}.

\section{Numerical Results}\label{sec:4}

In this section, we illustrate the convergence theory presented within
this paper with numerical experiments.

The domain~$\Omega$ was chosen to be the unit square $\Omega:=(0,1)^2$.
As mentioned in Section~\ref{sec:2}, the weak inf-sup-condition~\citecond{(S)} 
can be shown for the Taylor-Hood element only if at least one vertex of each element
is located in the interior of the domain~$\Omega$. As this is not satisfied for the standard 
decomposition of the unit square into two triangular elements, we have chosen the coarsest 
grid (grid level~$k=0$) to be a decomposition of the domain~$\Omega$ into $8$ triangles, 
cf. Fig.~\ref{fig:1}. The grid levels~$k=1,2,\ldots$ were constructed
by uniform refinement, i.e., every triangle was decomposed into four
subtriangles.

The desired velocity field (desired state) $u_D$ was chosen to be
\begin{equation*}
	u_D(\xi_1,\xi_2) := \left\{
				\begin{array}{ll}
					\left(\begin{array}{c}\xi_2-\tfrac12\\\tfrac12-\xi_1\end{array}\right)	& 
						\quad\mbox{for } \sqrt{\left(\xi_1-\tfrac12\right)^2+\left(\xi_2-\tfrac12\right)^2} < \frac{4}{5}\\
					0	& \quad\mbox{otherwise.}
				\end{array}
				\right.
\end{equation*}
The desired velocity field is visualized in the left-hand-side picture in both, Fig.~\ref{fig:3} and~\ref{fig:4}.

For solving the discretized KKT-system, we have used the proposed W-cycle multigrid method. We have applied
$\nu$ pre- and $\nu$ post-smoothing steps using the normal equation smoother. The
matrix $\mcL_k$ was chosen as follows
\begin{equation}\label{eq:defLnr}
        \mcL_k :=
        \left(
                \begin{array}{cccc}
                        \hat{A}_k \\
                        & \hat{S}_k \\
                        && \alpha^{-1} \hat{A}_k \\
                        &&& \alpha^{-1} \hat{S}_k
                \end{array}
        \right),
\end{equation}
where $\hat{A}_k := \mbox{diag }(M_{U,k}+\alpha^{1/2}K_{U,k})$ and $\hat{S}_k:=\alpha \,\mbox{diag }(D_k\hat{A}_k^{-1}D_k^T)$.
Here, $M_{U,k}$ and $K_{U,k}$ are the mass matrix and the stiffness matrix, representing the $L^2$-inner product
and the $H^1$-inner product in $U_k$, respectively. The matrix $D_k$ represents the bilinear form $d(u_k,p_k)=(\nabla\cdot u_k,p_k)_{L^2(\Omega)}$
on $U_k \times P_k$. Note that the matrix $\mcL_k$, introduced above, is spectrally equivalent to the matrix
$\mcL_k$, introduced in Section~\ref{sec:3}. Therefore, the choice proposed above is also covered by the
convergence theory. The damping parameter was chosen to be~$\tau=0.35$ for all grid levels~$k$ and all choices of~$\alpha$.

\begin{figure}[ht]%
\begin{center}
        \includegraphics[scale=.4]{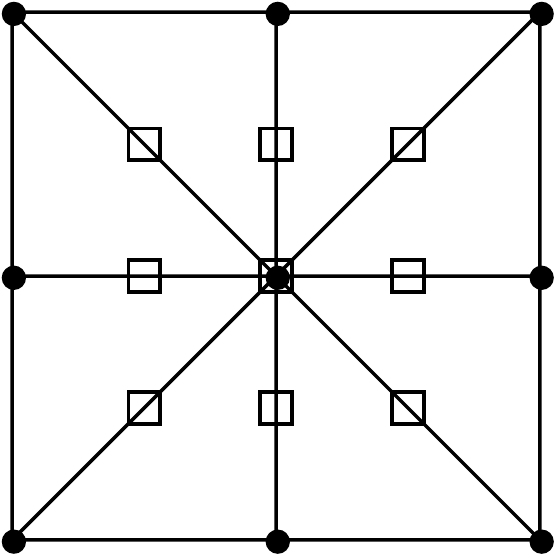}\qquad
        \includegraphics[scale=.4]{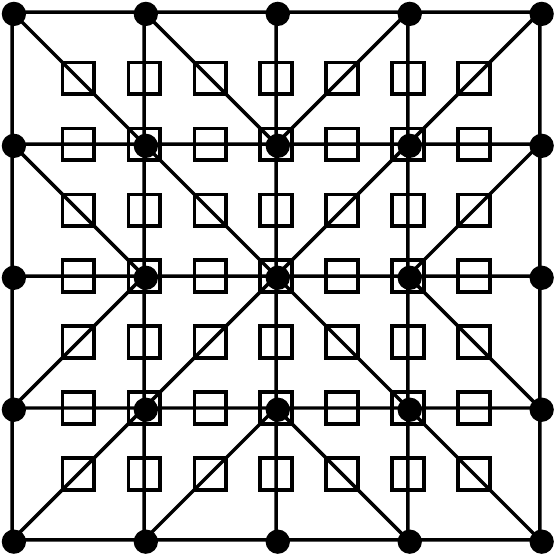}
        \caption{Discretization on grid levels $k=1$ and $k=2$, where the squares denote the degrees of freedom of (the components of) 
	$u$ and $\lambda$ are the the dots denote the degrees of freedom of $p$ and $\mu$}
        \label{fig:1}        
\end{center}
\end{figure}%
\begin{figure}[ht]%
\begin{center}
        \includegraphics[scale=.8]{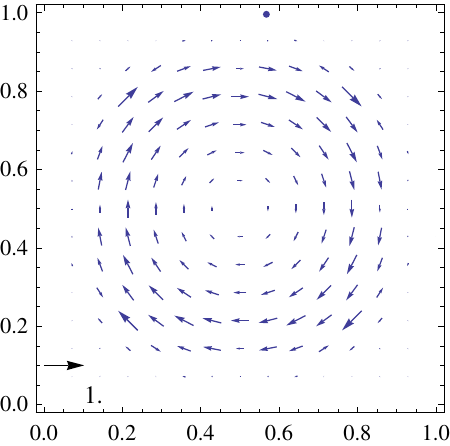}
        \includegraphics[scale=.8]{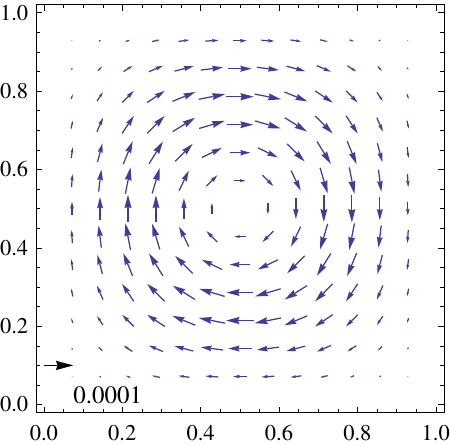}
        \includegraphics[scale=.8]{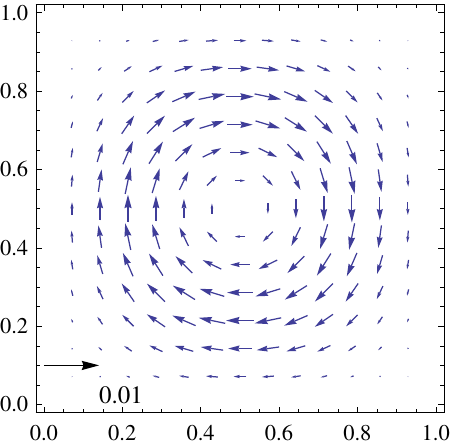}
        \caption{Desired velocity field $u_D$, optimal velocity field $u$ and optimal control $f$ for $\alpha = 1$ on grid level $k=3$}
        \label{fig:3}        
\end{center}
\end{figure}%
\begin{figure}[ht]%
\begin{center}
        \includegraphics[scale=.8]{desired}
        \includegraphics[scale=.8]{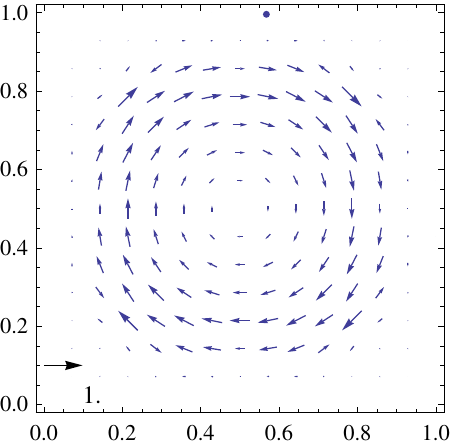}
        \includegraphics[scale=.8]{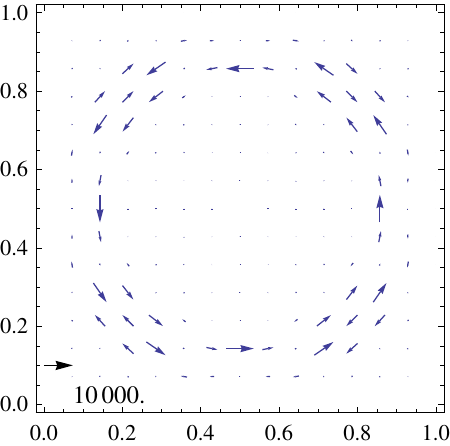}
        \caption{Desired velocity field $u_D$, optimal velocity field $u$ and optimal control $f$ for $\alpha = 10^{-12}$ on grid level $k=3$}
        \label{fig:4} 
\end{center}
\end{figure}%

The solution of the optimal control problem can be seen in Fig.~\ref{fig:3} and \ref{fig:4}. Note that the desired
velocity field is a general $L^2$-function (due to the jump) and therefore it cannot be reached by the (optimal) velocity 
field which is an $H^1$-function and therefore continuous. For the case $\alpha=1$, we observe the optimal velocity field and the control
to be rather smooth. (The control does not take large values). For small values of $\alpha$, like $\alpha=10^{-12}$,
the desired velocity field is approximated quite well, cf. Fig.~\ref{fig:4}. This is achieved
by rather large values of the control (high forces) which are concentrated on the region where the desired velocity field has its
jump. Forces have to be applied with the same orientation as the desired state, as well as with the opposite orientation. (In
the picture mainly the forces with opposite orientation can be seen.) As mentioned
in the introduction, we are interested in a fast linear solver which also works well for such small choices of $\alpha$.

The number of iterations and the convergence rate were measured as
follows: we start with $ x_k^{(0)}=0$ and measure the reduction of the error in each step
using the residual norm $\hnorm \cdot \hnorm_{2,k}$. The iteration was stopped when
the initial error was reduced by a factor of $\epsilon = 10^{-6}$. The convergence rates
$q$ is the mean convergence rate in this iteration, i.e.,
\begin{equation*}
        q = \left(\frac{\hnorm x_k^{(n)}-x_k^*\hnorm_{2,k}}{\hnorm x_k^{(0)}-x_k^*\hnorm_{2,k}}\right)^{1/n},
\end{equation*}
where $n$ is the number of iterations needed to reach the stopping
criterion. Here, $x_k^*$ is the exact solution and $x_k^{(i)}$ is the $i$-th iterate.

\begin{table}[ht]%
	\begin{center}
        \begin{tabular}{p{0.2cm}p{1.cm}p{0.2cm}p{1.cm}p{0.2cm}p{1.cm}p{0.2cm}p{1.cm}}
        \hline\noalign{\smallskip}
                \multicolumn{2}{l}{$\nu=1+1$} &
		 \multicolumn{2}{l}{$\nu=2+2$} &
		 \multicolumn{2}{l}{$\nu=4+4$} &
                 \multicolumn{2}{l}{$\nu=8+8$} \\
               $n$ & $q$& $n$ & $q$& $n$ & $q$& $n$ & $q$\\
        \noalign{\smallskip}\hline\noalign{\smallskip}
	        $61$&$0.796$&$32$&$0.647$&$21$&$0.507$&$15$&$0.390$\\
        \noalign{\smallskip}\hline\noalign{\smallskip}
        \end{tabular}
	\end{center}
        \caption{Number of iterations $n$ and convergence rate $q$ depending on $\nu=\nu_{pre}+\nu_{post}$, the 
	number of pre- and post-smoothing steps, on grid level $k=4$ for $\alpha=1$}
        \label{tab:0}        
\end{table}%
\begin{table}[ht]%
	\begin{center}
        \begin{tabular}{p{1.00cm}p{0.2cm}p{1.cm}p{0.2cm}p{1.cm}p{0.2cm}p{1.cm}p{0.2cm}p{1.cm}p{0.2cm}p{1.cm}}
        \hline\noalign{\smallskip}
               & \multicolumn{2}{l}{$\alpha=1$} &
		 \multicolumn{2}{l}{$\alpha=10^{-3}$} &
		 \multicolumn{2}{l}{$\alpha=10^{-6}$} &
                 \multicolumn{2}{l}{$\alpha=10^{-9}$} &
                 \multicolumn{2}{l}{$\alpha=10^{-12}$} \\
               & $n$ & $q$& $n$ & $q$& $n$ & $q$& $n$ & $q$  &$n$ & $q$\\
        \noalign{\smallskip}\hline\noalign{\smallskip}
         $k=3$ &$32$&$0.648$&$33$&$0.651$&$35$&$0.673$&$48$&$0.749$&$51$&$0.760$\\
         $k=4$ &$32$&$0.647$&$32$&$0.646$&$33$&$0.657$&$46$&$0.738$&$73$&$0.827$\\
         $k=5$ &$32$&$0.645$&$32$&$0.644$&$32$&$0.644$&$39$&$0.697$&$60$&$0.793$\\
         $k=6$ &$31$&$0.636$&$31$&$0.636$&$31$&$0.635$&$32$&$0.647$&$46$&$0.739$\\
         $k=7$ &$29$&$0.620$&$29$&$0.620$&$29$&$0.618$&$29$&$0.621$&$42$&$0.716$\\
        \noalign{\smallskip}\hline\noalign{\smallskip}
        \end{tabular}
	\end{center}
        \caption{Number of iterations $n$ and convergence rate $q$ for $\nu=2+2$ pre- and post-smoothing steps}
        \label{tab:1}        
\end{table}%
In Table~\ref{tab:0} we compare for a fixed grid level (level~$k=4$) and a fixed choice
$\alpha=1$ the convergence rates for several choices of $\nu$, the number of pre- and
post-smoothing steps. We see that the convergence rates behave approximately like~$\nu^{-1/2}$. 
This is consistent with the theory which
guarantees the convergence rate being bounded by $C\,\nu^{-1/2}$ as this only describes
the asymptotic behavior.

In Table~\ref{tab:1} we compare various grid levels~$k$ and choices of the parameter~$\alpha$.
Here, we have used a fixed choice of $\nu=2+2$ pre- and post-smoothing steps. First we observe
that the number of iterations seems to be well-bounded for all grid levels~$k$ which yields
an optimal convergence behavior. Moreover, we see that the number of iterations is also
well-bounded for a wide range of choices of the parameter~$\alpha$, i.e., we observe
also robust convergence as predicted by the convergence theory.

It has to be mentioned that for the model problem, also the (more efficient) V-cycle multigrid method
converges with rates comparable to the convergence rates of the W-cycle multigrid method. However,
the V-cycle is not covered by the convergence theory.

\section{Conclusions and Further Work}\label{sec:5}

In the present paper we have shown that the construction of an
all-at-once multigrid method for a Stokes control problem is possible. Here, a preconditioned 
normal equation smoother was chosen. The overall numerical complexity of this method seems to be
comparable to block-preconditioned MINRES iterations, cf., e.g., Table~4.2 
in~\cite{Zulehner:2010}, which shows the number of MINRES iterations needed. (Note that in each
MINRES step one multigrid cycle is applied to each component of the overall block-matrix, i.e., 
to the velocity~$u$, the pressure~$p$, the adjoined velocity~$\lambda$ and the adjoined pressure~$\mu$.)

One advantage of the all-at-once multigrid method, introduced in the present paper, is the fact
that an outer iteration is not necessary, the multigrid iteration is an linear iteration scheme
which can be directly applied to
solve the problem. As we could show the approximation property for a particular choice of
norms, the construction of other smoothers is of particular interest. The convergence rates we
have observed in this paper for a multigrid method with normal equation smoothing
are comparable with the convergence rates observed in~\cite{Takacs:Zulehner:2012} 
for a multigrid method with normal equation smoothing applied to an 
optimal control problem with elliptic state equation. For that problem we have
seen that other smoothers are available which lead to much faster convergence rates,
cf.~\cite{Takacs:Zulehner:2011} and others. Similar improvements were possible
for the generalized Stokes problem, cf.~\cite{Takacs:2013}. Therefore, it seems to be reasonable 
to construct faster smoothers also for the the Stokes control problem.

\textbf{Acknowledgements.} The author thanks Markus Kollmann for providing parts of the code used to compute 
the numerical results presented in this paper. Moreover, the support of the numerical analysis
group of the Mathematical Institute, University of Oxford, is gratefully acknowledged.

\bibliographystyle{plain}
\bibliography{literature}

\begin{thebibliography}{10}

\bibitem{Adams:Fournier}
R.~Adams and J.~Fournier.
\newblock {\em {Sobolev Spaces}}.
\newblock Academic Press, 2008.
\newblock 2nd ed.

\bibitem{Bercovier:Pironneau:1979}
M.~Bercovier and O.~Pironneau.
\newblock {Error estimates for finite element method solution of the Stokes
  problem in primitive variables}.
\newblock {\em Numerische Mathematik}, 33:211 -- 224, 1979.

\bibitem{Bergh:Loefstroem:1976}
J.~Bergh and J.~L\"ofstr\"om.
\newblock {\em {Interpolation Spaces, an Introduction}}.
\newblock Springer, Berlin, 1976.

\bibitem{Borzi:Schulz:2009}
A.~Borzi and V.~Schulz.
\newblock {Multigrid Methods for PDE Optimization}.
\newblock {\em SIAM Review}, 51:361 -- 395, 2009.

\bibitem{Brenner:1996}
S.C. Brenner.
\newblock {Multigrid methods for parameter dependent problems}.
\newblock {\em RAIRO, Mod\'elisation Math. Anal. Num\'er}, 30:265 -- 297, 1996.

\bibitem{Dauge:1988}
M.~Dauge.
\newblock {Elliptic boundary value problems on corner domains. Smoothness and
  asymptotics of solutions}.
\newblock {\em Lecture Notes in Mathematics, 1341. Berlin etc.:
  Springer-Verlag}, 1988.

\bibitem{Hackbusch:1985}
W.~Hackbusch.
\newblock {\em {Multi-Grid Methods and Applications}}.
\newblock Springer, Berlin, 1985.

\bibitem{Kellogg:Osborn:1976}
R.B. Kellogg and J.E Osborn.
\newblock {A regularity result for the Stokes problem in a convex polygon}.
\newblock {\em Journal of Functional Analysis}, 21(4):397--431, 1976.

\bibitem{Kollmann:Zulehner:2012}
M.~Kollmann and W.~Zulehner.
\newblock {A Robust Preconditioner for Distributed Optimal Control for Stokes
  Flow with Control Constraints}.
\newblock In Andrea Cangiani, Ruslan~L. Davidchack, Emmanuil Georgoulis,
  Alexander~N. Gorban, Jeremy Levesley, and Michael~V. Tretyakov, editors, {\em
  Numerical Mathematics and Advanced Applications 2011}, pages 771--779.
  Springer Berlin Heidelberg, 2013.

\bibitem{Olshanskii:2012}
M.~Olshanskii.
\newblock {Multigrid Analysis for the Time Dependent Stokes Problem}.
\newblock {\em Mathematics of Computation}, 81(277):57 -- 79, 2012.

\bibitem{Olshanskii:Peters:Reusken:2005}
M.~Olshanskii, J.~Peters, and A.~Reusken.
\newblock Uniform preconditioners for a parameter dependent saddle point
  problem with application to generalized stokes interface equations.
\newblock {\em Numerische Mathematik}, pages 159--191, 2005.

\bibitem{Pearson:2012}
J.~W. Pearson.
\newblock {On the Role of Commutator Arguments in the Development of
  Regularization-Robust Preconditioners for Stokes Control Problems}, 2012.
\newblock submitted.

\bibitem{Schoeberl:Simon:Zulehner:2010}
J.~Sch\"oberl, R.~Simon, and W.~Zulehner.
\newblock {A Robust Multigrid Method for Elliptic Optimal Control Problems}.
\newblock {\em SIAM~J.~on Numerical Analysis}, 49:1482 -- 1503, 2011.

\bibitem{Takacs:2013}
S.~Takacs.
\newblock {A multigrid method for the time-dependent Stokes problem}, 2013.
\newblock Submitted.

\bibitem{Takacs:Zulehner:2011}
S.~Takacs and W.~Zulehner.
\newblock {Convergence Analysis of Multigrid Methods with Collective Point
  Smoothers for Optimal Control Problems}.
\newblock {\em Computing and Visualization in Science}, 14(3):131--141, 2011.

\bibitem{Takacs:Zulehner:2012}
S.~Takacs and W.~Zulehner.
\newblock {Convergence analysis of all-at-once multigrid methods for elliptic
  control problems under partial elliptic regularity}.
\newblock {\em SIAM~J.~on Numerical Analysis}, 2012.
\newblock Accepted.

\bibitem{Verfuerth:1984}
R.~Verf\"urth.
\newblock {Error estimates for a mixed finite element approximation of the
  Stokes equations}.
\newblock {\em RAIRO}, 18:175 -- 182, 1984.

\bibitem{Zulehner:2010}
W.~Zulehner.
\newblock {Non-standard Norms and Robust Estimates for Saddle Point Problems}.
\newblock {\em SIAM~J.~on Matrix Anal. \& Appl}, 32:536 -- 560, 2011.

\end{thebibliography}

\vspace{-3cm}
\parbox{20cm}{\vspace{4cm} \textbf{The original publication is available at www.springerlink.com:}\\
		\url{http://link.springer.com/article/10.1007\%2Fs00211-014-0674-5} \vspace{-3cm}}

\end{document}